\documentclass[12pt]{amsart}
\usepackage{srcltx}
\usepackage{amsthm, amsmath, amssymb, paralist, calc}
\usepackage{a4}
\newtheorem{theorem}{Theorem}[]

\newtheorem{lemma}[theorem]{Lemma}

\newtheorem{corollary}[theorem]{Corollary}
\newtheorem{conjecture}[theorem]{Conjecture}

{\theoremstyle{definition}

}

\def \St {\mathrm{St}}

\DeclareMathOperator{\proj}{proj}

\allowdisplaybreaks

\renewcommand{\O}{\Omega}

\begin{document}

\title{Completely reducible subcomplexes of spherical buildings}

\author{Chris Parker} \author{ Katrin Tent}
\address{Chris Parker\\
School of Mathematics\\
University of Birmingham\\
Edgbaston\\
Birmingham B15 2TT\\
United Kingdom} \email{c.w.parker@bham.ac.uk}

\address{Katrin Tent\\Mathematisches 
Institut \\
Universit\"at M\"unster\\
Einsteinstrasse 62\\
48149 M\"unster\\
Germany 
}
\email{tent@uni-muenster.de}

\maketitle

In 2005 Serre in \cite{Serre} introduced the notion of complete reducibility in spherical buildings. He went on to point out the following conjecture \cite[Conjecture 2.8]{Serre} which he attributes to Tits from the 1950's.

\begin{conjecture}[Tits' Centre Conjecture] \label{conj1} Suppose that  $\Delta$ is a spherical building and $\O$ is a convex subcomplex of $\Delta$. Then (at least) one of the following holds:
\begin{enumerate}
\item [(a)] for each simplex $A$ in $\O$, there is a simplex $B$ in $\O$ which is opposite to $A$ in $\Delta$; or \item[(b)] there exists a nontrivial simplex $A'$ in $\O$ fixed by any automorphism of $\Delta$ stabilizing $\Omega$.
\end{enumerate}
\end{conjecture}
If possibility (a) in the conjecture arises we say that $\Omega$ is \emph{completely reducible} and if (b) is the case, then the simplex $A'$ is called a \emph{centre} of $\Omega$. If alternative (a) holds then $\Omega$ is a possibly thin subbuilding of $\Delta$ (see \cite{Kramer}).

If $G$ is an algebraic group with associated building $\Delta$, then a subgroup $H$ of $G$ is called \emph{completely reducible} provided that whenever it is a subgroup of a parabolic subgroup  of $G$ it  is contained in a Levi complement of that parabolic subgroup. In this case, the convex subcomplex of $\Delta$ fixed by $H$ is completely reducible. Conversely if the subcomplex of $\Delta$ fixed by a subgroup $H$ of a parabolic subgroup of $G$ is completely reducible, then so is $H$. This relationship between complete reducibility of subcomplexes of the building and completely reducible subgroups of parabolic subgroups has lead to a  source of fruitful research of which we particularly mention \cite[Theorem 3.1]{Bates} in which they prove the conjecture in the case that $\Omega$ is the fixed point set of some subgroup $H$.

In the more general setting,  for the classical buildings and  buildings of rank 2 the conjecture was proved by M\"uhlherr and Tits \cite{Muhlherr-Tits} in 2006. For buildings of exceptional type $\mathrm E_6$ ,$\mathrm E_7$ and $\mathrm E_8$ the conjecture has been  proved by Leeb and Ramos Cuevas \cite{LC,C} using, in part, some of the observations presented in this paper. They also include the proof of the conjecture for buildings of type $\mathrm F_4$, which was first
presented by the authors at a meeting in Oberwolfach in January 2007 \cite{PT}.
All of the investigations of the Centre Conjecture have  used the lemma
of Serre's \cite{Serre} which states that $\O$ is completely reducible if \emph{every} vertex of $\O$ has an opposite. For chamber complexes,
we can prove the following stronger assertion and thereby obtain a very short proof of  the Centre Conjecture for convex chamber subcomplexes of classical buildings.

\begin{theorem}\label{vertextype} Let $\Delta$ be an irreducible spherical building of type $(W,I)$.
Let $\Omega$ be a convex chamber subcomplex of $\Delta$.
If for some $k\in I$ every vertex of type $k$ in  $\Omega$ has an opposite in $\Omega$, then
$\Omega$ is completely reducible.
\end{theorem}

Notice that the hypothesis that $\Delta$ is irreducible in Theorem~\ref{vertextype} may not be dropped as is easily seen by taking a product of two buildings and choosing a convex subcomplex which is completely reducible in one factor and has a centre in the second factor.
 Our notation follows \cite{TitsLN}. So  given a simplex $R$ of type $J \subseteq I$, the collection of all simplices containing $R$ form a building $\mathrm {St}R$ of type $(W_{I\setminus J}, I\setminus J)$.
 Of particular importance to us are the \emph{projection maps}: given simplices $R$ and $S$,  $\proj_R(S)$ is the unique simplex of $\mathrm{St}R$  which is contained in every shortest gallery from $S$ to $R$ (see \cite[Proposition 2.29]{TitsLN}) and is called the \emph{projection} of $S$ to $R$. Note that if $\Omega$ is a convex subcomplex of $\Delta$ then, for all simplices $R$ and $S$ in $\O$, we have $\proj_RS \in \Omega$ and this is the crucial property of convexity that we use in the proof of Theorem~\ref{vertextype}. We refer the reader to \cite[2.30 and 2.31]{TitsLN} for many properties of projection maps.  Two chambers in $\Delta$ are \emph{opposite} in $\Delta$  provided their convex hull is an apartment of $\Delta$.  Two simplices $R$ and $R'$  of $\Delta$ are \emph{opposite} in $\Delta$ if every chamber of $\mathrm {St} R$ has an opposite  in  $\mathrm {St}  R'$.

\begin{lemma}\label{opp}
Suppose that $x$ and $y$ are opposite chambers in $\Delta$. Let $\Sigma $ be the  convex hull of $x$ and $y$ in $\Delta$ and   $R$ be a simplex  in
$\Sigma$. Then $\proj_R(x)$ and $\proj_R(y)$ are opposite in $\mathrm {St}R$.
\end{lemma}

\begin{proof}
 Set $x_1=\proj_R(x)$ and $y_1=\proj_R(y)$. Then $x_1$ and $y_1$ are chambers by  \cite[Proposition 2.29]{TitsLN}. Let $z$ be
opposite $x_1$ in $\mathrm{St}R$. Then we have $\mathrm{dist}(x,z)= \mathrm{dist}(x,x_1)+\mathrm{dist}(x_1,z)$ and
$\mathrm{dist}(y,z)= \mathrm{dist}(y,y_1)+ \mathrm{dist}(y_1,z)$ by \cite[2.30.6]{TitsLN}. Therefore $\mathrm{dist}(x,y) =
\mathrm{dist}(x,x_1)+\mathrm{dist}(x_1,z)+\mathrm{dist}(y_1,z)+ \mathrm{dist}(y,y_1)$ as every chamber of $\Sigma$
is on a shortest gallery between $x$ and $y$ by \cite[2.35 (iv)]{TitsLN}. On the other hand, as
$x_1$ and $z$ are opposite in $\mathrm{St}R$,  $\mathrm{dist}(x_1,y_1) \le \mathrm{dist}(x_1,z)$ and so
\begin{eqnarray*}\mathrm{dist}(x,y)&=& \mathrm{dist}(x,x_1)+\mathrm{dist}(x_1,y_1)+ \mathrm{dist}(y_1,y)\\& \le& \mathrm{dist}(x,x_1)+\mathrm{dist}(x_1,z)+
\mathrm{dist}(y,y_1).\end{eqnarray*} It follows that $\mathrm{dist}(y_1,z)=0$ and hence $z=y_1$ as
claimed.
\end{proof}

  The following observation is especially important to us.

\begin{corollary}\label{cor1} Suppose that $R$, $X$ and $Y$ are simplices  in the apartment $\Sigma$   with $X$ opposite $Y$.
 Then either \begin{enumerate} \item [(a)]$\proj_R(X)$ is opposite  $\proj_R(Y)$ in $\St R$; or \item[(b)] $R=\proj_R(X)=\proj_R(Y)$. \end{enumerate}
\end{corollary}

\begin{proof}
We can pair the chambers containing  $X$ and $Y$ into opposite pairs $(x,y)$.
Then $\proj_R(x)$ is opposite $\proj_R(y)$ in $\St R$ by Lemma~\ref{opp}.  This means every
chamber of $\proj_R(X)$ has an opposite in $\St R$ contained in
$\proj_R(Y)$.
\end{proof}

We can now prove Theorem~\ref{vertextype}. So suppose that $\Omega$ is a convex  chamber subcomplex of $\Delta$. We recall that  $\Omega$ is a subcomplex, means that if a simplex is in $\Omega$ then so are all of its faces and $\Omega$ is a chamber complex means that every simplex is contained in a chamber.
 We repeatedly use the fact that, as $\Omega$ is convex,  projections between simplices  of $\Omega$ are contained in $\Omega$.

By hypothesis, we may choose $J\subseteq I$ maximally so that every simplex of type $J$ in $\Omega$ has an opposite in $\Omega$.  It suffices to show that $J= I$,   as, if  a chamber has an opposite, then so does every face of that chamber. So suppose that $J \neq I$. Since $\Delta$ is irreducible there is  $i\in I\setminus J$  such that $i$ is a neighbour of some $j\in J$ in the Dynkin diagram of $\Delta$.

Let $z$ be of type $J\cup \{i\}$ in $\Omega$,   $x_0$ be the face of $z$ of type $J$, $\ell$ the vertex of $z$  of type~$i$ and let $C_0$ be a chamber of $\Omega$ containing $z$.
We will construct an opposite for $z$. 

Let  $p$ be a maximal face of $C_0$ with missing vertex of type $j$ and $x_0^o$ be an opposite of $x_0$ in $\Omega$. Then $\ell$ is a vertex of $p$.
Put $C_0'=\proj_{x_0^o} C_0$ and
 $C_1=\proj_{p}C_0'$. Then, by Corollary~\ref{cor1}, $C_0= \proj_{p}(x_0) \neq C_1$.  Let $x_1$ be the face of $C_1$  of type $J$.
So $x_1\neq x_0$ and setting  $y_0=\proj_{x_1}x_0$ we see that, as the reflections corresponding to $i$ and $j$ do not commute, $y_0$ has   $x_1$ as a face and  $\ell$ as a vertex.
We will first find an opposite of the simplex~$y_0$.

Let $y_1=\proj_{x_1}x_0^o$, so $y_1$ and $y_0$ are opposite in  $\St x_1$ by Corollary~\ref{cor1}.
Let $x_1^o$ be opposite $x_1$. By  \cite[Proposition 3.29]{TitsLN}, we have $y_2=\proj_{x_1^0}(y_1)$
is opposite $y_0$. Since $y_0$ contains the vertex $\ell$, $y_2$ has an opposite of $\ell$ as a vertex and this is contained in $\Omega$.

In order to find an opposite for the simplex $z$, notice that
$\proj_{\ell}x_0=z$. Let $z_1=\proj_{\ell}x_0^o$, so $z_1$ and $z$ are opposite in   $\St \ell$ by Corollary~\ref{cor1}.
Using \cite[Proposition 3.29]{TitsLN} again, the projection of $z_1$ to the opposite of $\ell$ in $\St y_2$ now yields the required opposite of $z$ in $\Omega$.
\qed
\begin{corollary}
The Centre Conjecture holds for convex chamber subcomplexes of  irreducible spherical buildings of classical type.
\end{corollary}

\begin{proof} For buildings of type $\mathrm A_n,\mathrm B_n$, $\mathrm C_n$ and $\mathrm D_n$, we identify the simplices of $\Delta$ with  flags of subspaces (singular subspaces, isotropic subspaces) in the appropriate vector spaces.  We then consider the vertices of $\Delta$ corresponding to $1$-dimensional  subspaces (for $\mathrm A_n$) and $1$-dimensional isotropic/singular  subspaces in the other cases and call them type 1 vertices.

Since $\Omega$ is a chamber subcomplex, $\Omega$ contains vertices of every type.  If every type 1 vertex has an opposite in  $\Omega$, then $\Omega$ is completely reducible by Theorem~\ref{vertextype}. So we  suppose that this is not the case and aim to identify a centre.

Suppose that $\Delta$ has type $\mathrm A_n$ and assume that some type $1$ vertex $w$ of $\Omega$ does not have an opposite in $\Omega$.
Then $w$ is contained in all the  hyperplanes of $\Omega$. Thus the intersection of all
hyperplanes of $\Omega$
is the required centre.

Suppose that  $\Delta$ has type $\mathrm B_n, \mathrm C_n$ or $\mathrm D_n$. Then a vertex of type $1$ in $\Omega$ has no opposite in $\Omega$ if and only if it is collinear with every other vertex of type $1$  in $\Omega$.
Hence the set of all vertices  of type 1 in  $\Omega$ having no opposite span 
a totally isotropic (singular) subspace, and this is  the centre.
\end{proof}

\end{document}